\pdfoutput=1 
\documentclass[10pt]{amsart}

\title{Improved K\"unneth Tricks}
\author{Amnon Yekutieli}
\address{Department of Mathematics,
Ben Gurion University, Be'er Sheva 84105, Israel.
\newline \indent \textup{\textit{Email}:
\href{mailto:amyekut@gmail.com}{\rm \scriptsize
\nolinkurl{amyekut@gmail.com}},
\textit{Web}: \rm \scriptsize
\url{https://sites.google.com/view/amyekut-math/home}}}
\date{2 August 2023}

\usepackage{tikz}
\usepackage{tikz-cd}
\usetikzlibrary{arrows}
\tikzcdset{arrow style=tikz,
diagrams={>={Stealth[round,length=4pt,width=6pt,inset=3pt]}}}
\usepackage{graphicx}
\usepackage{hyperref}
\hypersetup{colorlinks=false}

\usepackage[T1]{fontenc}
\usepackage[theoremfont]{newpxtext}
\usepackage[scaled=1.00]{roboto}
\usepackage[varbb, varg, smallerops]{newpxmath}
\thanks{
Primary: 16E45. Secondary: 16E35, 18G10, 18E30.}


\newtheorem{thm}[equation]{Theorem}

\theoremstyle{definition}

\newcommand{\iso}{\xrightarrow{
\smash{\raisebox{-0.5ex}{\ensuremath{\scriptstyle \simeq  \mspace{2mu}}}}}}

\newcommand{\xar}{\xrightarrow}

\newcommand{\opn}{\operatorname}
\newcommand{\cat}[1]{\operatorname{\mathsf{#1}}}

\newcommand{\cd}{\mspace{1.8mu}{\cdotB}\mspace{2.0mu}}

\newcommand{\rmitem}[1]{\item[\textrm{(#1)}]}

\newcommand{\mrm}[1]{\mathrm{#1}}

\renewcommand{\th}{\theta}

\newcommand{\K}{\mathbb{K}}

\newcommand{\Z}{\mathbb{Z}}

\newcommand{\ot}{\otimes}

\renewcommand{\d}{\mathrm{d}}

\newcommand{\msp}[1]{\mspace{#1 mu}}

\begin{document}

\begin{abstract}
The K\"unneth trick is a formula for the top cohomology of the derived tensor
product of two complexes of modules over a ring.

In this note we present two improvements of this formula. The first
improved K\"unneth trick is a formula for the top cohomology
of the plain tensor product of two DG modules over a nonpositive DG ring.
The second trick handles the derived tensor product of two DG
modules over a nonpositive DG ring. The proofs are elementary.
\end{abstract}

\maketitle


The {\em Ku\"nneth trick} is a useful formula for the top cohomology of the
derived tensor product $M \ot^{\mrm{L}}_A N$ of complexes of modules over a
ring $A$. It was stated without a proof as Lemma 13.1.36 in our book \cite{Ye2},
and earlier we stated it as Lemma 2.1 in the paper \cite{Ye1},
with an indirect proof relying on the convergent K\"unneth spectral sequence.

In this short paper we have two improved Ku\"nneth tricks. Theorem
\ref{thm:100} gives a formula for the top cohomology of
$M \ot_A N$, where $A$ is a {\em nonpositive DG ring}, $M$ is a bounded above
{\em right DG $A$-module}, and {\em $N$ is a bounded above left DG $A$-module}.
Theorem \ref{thm:101} does the same for the
derived tensor product $M \ot^{\mrm{L}}_A N$.

In this paper we follow the definitions and notation of the
book \cite{Ye2}. Chapter 3 of this book has a detailed treatment of DG
rings and DG modules. Chapter 7 of the book is about derived categories of DG
modules, Chapter 8 talks about derived functors, and Section 12.3 deals with
$M \ot^{\mrm{L}}_A N$.
In the paragraphs below we provide a very brief recollection of this material.

We fix a commutative base ring $\K$ (possibly $\K = \Z$), and a
{\em nonpositive central DG $\K$-ring}
$A = \bigoplus_{i \leq 0} A^i$ (possibly noncommutative).
We let $\bar{A} := \opn{H}^0(A)$, which is a central $\K$-ring. There is
a canonical DG $\K$-ring homomorphism $A \to \bar{A}$.

Here are some useful special cases.
One is when $A$ is a {\em central $\K$-ring}, i.e.\ $A = A^0$; in traditional
language such $A$ is called a unital associative $\K$-algebra.
Another special case is when $A$ is a {\em weakly commutative DG ring}
(see \cite[Definition 3.3.4]{Ye2}), and in this case we can take $\K = A^0$.
The conjunction of these two cases is when $A = A^0$ is a commutative ring.

All DG $A$-modules are left DG modules by default, and there are no boundedness
conditions, except when specified.
The DG category of (left) DG $A$-modules is denoted by $\cat{C}_{}(A)$.
The strict subcategory $\cat{C}_{\mrm{str}}(A)$ has the same objects as
$\cat{C}_{}(A)$, but its morphisms are the
strict homomorphisms, namely the degree $0$ homomorphisms that commute with the
differentials, and it is a $\K$-linear abelian category.
Observe that for $M \in \cat{C}_{}(A)$, its cohomologies
$\opn{H}^i(M)$ are $\bar{A}$-modules.

Given an integer $i_0$, we let
$\cat{C}^{\leq i_0}_{\mrm{str}}(A)$ be
the full subcategory of $\cat{C}_{}(A)$ on the DG modules $M$ that are
concentrated in degrees $\leq i_0$, i.e.\
$M = \bigoplus_{i \leq i_0} M^i$.
We indentify the cateogy $\cat{M}(A)$ of $A$-modules with the full subcategory
of $\cat{C}_{}(A)$ on the DG modules $M$ that are concentrated in degree
$0$.

The derived category of DG $A$-modules is $\cat{D}(A)$. It is a $\K$-linear
triangulated category.
For an integer $i_0$, we let
$\cat{D}^{\leq i_0}(A)$ be
the full subcategory on the complexes $M$ whose {\em cohomology} is
concentrated in degrees $\leq i_0$. Of course $\cat{D}^{\leq i_0}(A)$
is not a triangulated subcategory of $\cat{D}(A)$.

Right DG $A$-modules are treated a left DG modules over the opposite DG ring
$A^{\mrm{op}}$. These are the objects of the DG category
$\cat{C}_{}(A^{\mrm{op}})$ and the triangulated category
$\cat{D}_{}(A^{\mrm{op}})$.

Given DG modules
$M \in \cat{C}^{\leq i_0}_{}(A^{\mrm{op}})$
and
$N \in \cat{C}^{\leq j_0}_{}(A)$,
their tensor product $M \ot_A N$ belongs to
$\cat{C}^{\leq i_0 + j_0}_{}(\K)$.
Similarly, given DG modules
$M \in \cat{D}^{\leq i_0}(A^{\mrm{op}})$
and
$N \in \cat{C}^{\leq j_0}(A)$,
their derived tensor product $M \ot^{\mrm{L}}_A N$ belongs to
$\cat{D}^{\leq i_0 + j_0}(\K)$.

For a DG module $N \in \cat{C}^{\leq j_0}_{}(A)$,
the elements of $N^{j_0}$ are all cocycles.
Given an element $n \in N^{j_0}$, we denote by $[n]$ the cohomology class of
$n$ in $\opn{H}^{j_0}(N)$.
Similarly for $M \in \cat{C}^{\leq i_0}_{}(A^{\mrm{op}})$ and
$m \in M^{i_0}$.

\begin{thm}[Plain K\"unneth Trick] \label{thm:100}
Let $A$ be a nonpositive central DG $\K$-ring, and let
$M \in \cat{C}^{\leq i_0}_{}(A^{\mrm{op}})$
and
$N \in \cat{C}^{\leq j_0}_{}(A)$
for some integers $i_0$ and $j_0$.
Then there is a unique isomorphism of $\K$-modules
\[ \th_{M, N} : \opn{H}^{i_0}(M) \ot_{\bar{A}}
\opn{H}^{j_0}(N) \iso
\opn{H}^{i_0 + j_0}(M \ot_{A} N) \]
such that
$\th_{M, N} ([m] \ot [n]) =  [m \ot n ]$
for all $m \in M^{i_0}$  and $n \in N^{j_0}$.
The isomorphism $\th_{M, N}$ is functorial in $M$ and $N$.
\end{thm}

\begin{proof}
After translating the complexes $M$ and $N$, we may assume that
$i_0 = j_0 = 0$.

There is an exact sequence of $\K$-modules
\[ (M \ot_{A} N)^{-1} \xar{\msp{5} \d \msp{5}} (M \ot_{A} N)^{0}
\xar{\msp{5} \pi \msp{5}} \opn{H}^{0}(M \ot_{A} N) \to 0 , \]
where $\d$ is the differential of the DG $\K$-module
$M \ot_A N$, and
$\pi(m \ot n) := [m \ot n]$.
By degree considerations, it is not hard to see that the obvious homomorphism
$M^0 \ot_{A^0} N^0 \to (M \ot_{A} N)^{0}$ is bijective, and the homomorphism
\[ (M^{-1} \ot_{A^0} N^{0}) \oplus (M^{0} \ot_{A^0} N^{-1}) \to
(M \ot_{A} N)^{-1} \]
is surjective. (The second homomorphism might fail to be injective: for
elements $a \in A^{-1}$, $m \in M^{0}$ and $n \in N^{0}$,
the element $(m \cd a) \ot n - m \ot (a \cd n)$
could be nonzero in the source, but it is always zero in the target.)
Therefore we get an exact sequence of $\K$-modules
\begin{equation} \label{eqn:105}
(M^{-1} \ot_{A^0} N^{0}) \oplus (M^{0} \ot_{A^0} N^{-1})
\xar{\msp{5} \phi \msp{5}}
M^{0} \ot_{A^0} N^{0}
\xar{\msp{5} \pi \msp{5}}
\opn{H}^{0}(M \ot_{A} N) \to 0 ,
\end{equation}
where
$\phi := (\d_M \ot \opn{id}_N) \oplus (\opn{id}_M \ot \msp{2} \d_N)$.

We have an exact sequence of $A^0$-modules
\[  N^{-1} \to N^{0} \xar{\msp{5} \pi_N \msp{5}} \opn{H}^{0}(N) \to 0 . \]
Applying $\opn{H}^{0}(M) \ot_{A^0} (-)$ to it gives the exact sequence of
$\K$-modules
\begin{equation} \label{eqn:109}
\opn{H}^{0}(M) \ot_{A^0} N^{-1} \to \opn{H}^{0}(M) \ot_{A^0} N^{0} \to
\opn{H}^{0}(M) \ot_{A^0} \opn{H}^{0}(N) \to 0 .
\end{equation}
The surjection $\pi_M : M^0 \to \opn{H}^{0}(M)$ allows us to replace
(\ref{eqn:109})
with this exact sequence:
\begin{equation} \label{eqn:130}
M^0 \ot_{A^0} N^{-1} \to \opn{H}^{0}(M) \ot_{A^0} N^{0} \to
\opn{H}^{0}(M) \ot_{A^0} \opn{H}^{0}(N) \to 0 .
\end{equation}

Next there is an exact sequence of $(A^0)^{\mrm{op}}$-modules
\[  M^{-1} \to M^{0} \xar{\msp{5} \pi_M \msp{5}} \opn{H}^{0}(M) \to 0 . \]
Applying $(-) \ot_{A^0} N^0$ to it gives the exact sequence of
$\K$-modules
\begin{equation} \label{eqn:131}
M^{-1} \ot_{A^0} N^0 \to M^{0}\ot_{A^0} N^0 \to
\opn{H}^{0}(M)\ot_{A^0} N^0 \to 0 .
\end{equation}
By replacing the term
$\opn{H}^{0}(M)\ot_{A^0} N^0$ in (\ref{eqn:130}) with
$M^0 \ot_{A^0} N^0$, and using the exact sequence (\ref{eqn:131}),
we obtain this exact sequence:
\begin{equation} \label{eqn:133}
(M^{-1} \ot_{A^0} N^{0}) \oplus (M^{0} \ot_{A^0} N^{-1})
\xar{\msp{3} \phi \msp{3}}
M^{0} \ot_{A^0} N^{0}
\xar{\msp{3} \pi_M \ot_{A^0} \pi_N \msp{3}}
\opn{H}^{0}(M) \ot_{A^0} \opn{H}^{0}(N) \to 0 ,
\end{equation}
where $\pi_M$, $\pi_N$ and $\phi$ are as above.

By comparing the exact sequences (\ref{eqn:105}) and
(\ref{eqn:133}) we obtain an isomorphism
\[ \opn{H}^{0}(M) \ot_{A^0} \opn{H}^{0}(N) \iso
\opn{H}^{0}(M \ot_{A} N) \msp{2} , \msp{5}
[m] \ot [n] \mapsto [m \ot n ] . \]
Composing this with the canonical isomorphism
$\opn{H}^{0}(M) \ot_{\bar{A}} \opn{H}^{0}(N) \iso
\opn{H}^{0}(M) \ot_{A^0} \opn{H}^{0}(N)$
yields the desired isomorphism $\th_{M, N}$.
The uniqueness and functoriality are clear.
\end{proof}

Recall that for
$M \in \cat{D}(A^{\mrm{op}})$ and
$N \in \cat{D}(A)$
there is the bifunctorial morphism
\[ \eta^{\mrm{L}}_{M, N} : M \ot^{\mrm{L}}_A N \to M \ot_A N \]
in $\cat{D}(\K)$, which is part of the left derived functor.

\begin{thm}[Derived Ku\"nneth Trick] \label{thm:101}
Let $A$ be a nonpositive central DG $\K$-ring,
and let
$M \in \cat{D}^{\leq i_0}_{}(A^{\mrm{op}})$ and
$N \in \cat{D}^{\leq j_0}_{}(A)$
for some integers $i_0$ and $j_0$.
Then there is a unique isomorphism of $\K$-modules
\[ \th^{\mrm{der}}_{M, N} :
\opn{H}^{i_0}(M) \ot_{\bar{A}} \opn{H}^{j_0}(N) \iso
\opn{H}^{i_0 + j_0}(M \ot^{\mrm{L}}_A N) \]
with the following two properties:
\begin{itemize}
\rmitem{i} $\th^{\mrm{der}}_{M, N}$ is functorial in $M$ and $N$.

\rmitem{ii} If $M \in \cat{C}^{\leq i_0}_{}(A^{\mrm{op}})$
and $N \in \cat{C}^{\leq j_0}_{}(A)$, then the diagram
\[ \begin{tikzcd} [column sep = 8ex, row sep = 6ex]
\opn{H}^{i_0}(M) \ot_{\bar{A}} \opn{H}^{j_0}(N)
\ar[r, "{\th^{\mrm{der}}_{M, N}}"]
\ar[dr, "{\th^{}_{M, N}}"']
&
\opn{H}^{i_0 + j_0}(M \ot^{\mrm{L}}_A N)
\ar[d, "{\opn{H}^{i_{0} + j_0}(\eta^{\mrm{L}}_{M, N})}"]
\\
&
\opn{H}^{i_0 + j_0}(M \ot^{}_A N)
\end{tikzcd} \]
is commutative.
\end{itemize}
\end{thm}

\begin{proof}
As in the previous proof, we can assume that $i_0 = j_0 = 0$.
By smart truncation of $N$ we can also assume that
$N \in \cat{C}^{\leq 0}_{}(A)$.

Let $\rho : P \to M$ be a semi-free resolution in
$\cat{C}_{\mrm{str}}(A^{\mrm{op}})$
such that $\opn{sup}(P) = \opn{sup}(\opn{H}(M))$.
This exists by \cite[Corollary 11.4.27]{Ye2}.
The complex $P$ belongs to $\cat{C}^{\leq 0}_{}(A^{\mrm{op}})$.

The morphism
$\eta^{\mrm{L}}_{P, N} : P \ot^{\mrm{L}}_A N \to P \ot_A N$
in $\cat{D}(A)$ is an isomorphism.
To satisfy condition (ii) we have no choice but to define
$\th^{\mrm{der}}_{P, N}$ to be the unique isomorphism such that
$\opn{H}^0(\eta^{\mrm{L}}_{P, N}) \circ \th^{\mrm{der}}_{P, N} =
\th_{P, N}$.

Next, in order to satisfy condition (i), we must define
$\th^{\mrm{der}}_{M, N}$ to be the unique isomorphism making this diagram
in $\cat{M}(\K)$ commutative:
\[ \begin{tikzcd} [column sep = 14ex, row sep = 6ex]
\opn{H}^{0}(P) \ot_{\bar{A}} \opn{H}^{0}(N)
\ar[d, "{\th^{\mrm{der}}_{P, N}}"', "{\simeq}"]
\ar[r, "{\opn{H}^0(\rho) \msp{3} \ot_{\bar{A}} \msp{3}
\opn{H}^0(\opn{id}_N)}", "{\simeq}"']
&
\opn{H}^{0}(M) \ot_{\bar{A}} \opn{H}^{0}(N)
\ar[d, "{\th^{\mrm{der}}_{M, N}}", "{\simeq}"']
\\
\opn{H}^{0}(P  \ot^{\mrm{L}}_A N)
\ar[r, "{\opn{H}^0(\opn{Q}(\rho) \msp{3} \ot^{\mrm{L}}_A \msp{3} \opn{id}_N)}",
"{\simeq}"']
&
\opn{H}^{0}(M  \ot^{\mrm{L}}_A N)
\end{tikzcd} \]

It is routine to verify that conditions (i) and (ii) hold.
\end{proof}


\end{document}